\documentclass[12pt]{article}
\usepackage{amsmath,amsthm,amsfonts,amssymb,mathrsfs,epsfig,bm}
\usepackage[usenames]{color}
\usepackage[colorlinks=true]{hyperref}

\parskip        \smallskipamount
\newtheorem{theorem}{Theorem}%[section]
\newtheorem{lemma}{Lemma}

\theoremstyle{remark}
\newtheorem{remark}{Remark}

\newtheorem{example}{\bf Example}
\def\R{\mathbb{R}}

\def\P{\mathbb{P}}
\def\E{\mathbb{E}}
\def\FF{\mathscr{F}}

\def\SS{\mathscr{S}}
\def\PP{\mathscr{P}}

\def\NN{\mathcal N}

\renewcommand{\phi}{\varphi}
\renewcommand{\epsilon}{\varepsilon}

\newcommand{\1}{{\text{\Large $\mathfrak 1$}}}

\definecolor{mygray}{gray}{0.9}
\definecolor{deeppink}{RGB}{255,20,147}
\definecolor{mygreen}{rgb}{0.05, 0.576, 0.03}
\definecolor{myred}{rgb}{0.768, 0.09, 0.09}

\long\def\symbolfootnote[#1]#2{\begingroup
\def\thefootnote{\fnsymbol{footnote}}\footnote[#1]{#2}\endgroup}
\newcommand{\keywords}[1]{ \noindent {\footnotesize
             {\small \em Keywords and phrases.} {\sc #1} } }
\newcommand{\ams}[2]{  \noindent {\footnotesize
             {\small \em AMS {\rm 2000} subject classifications.
             {\rm Primary {\sc #1}; secondary {\sc #2}} } } }

\def\MM{\mathscr{M}}

\usepackage{MnSymbol}

\def\NN{\mathscr N}
\def\cP{\mathcal P}
\def\cQ{\mathcal Q}
\def\hcQ{\widehat{\cQ}}
\begin{document}

\title{\Large \bf On a representation theorem for
finitely exchangeable random vectors}
\author{{Svante Janson\footnote{svante.janson@math.uu.se, http://www2.math.uu.se/$\sim$svante,
partly supported by the Knut and Alice Wallenberg Foundation} 
\hspace{1cm} Takis Konstantopoulos\footnote{takiskonst@gmail.com, http://www2.math.uu.se/$\sim$takis, 
supported by Swedish Research Council grant 2013-4688}} \\ Uppsala University
\and  Linglong Yuan \footnote{yuanlinglongcn@gmail.com, http://linglongyuan.weebly.com}\\ Johannes-Gutenberg-Universit\"at Mainz}
\date{}
\maketitle

%\author{Svante Janson\fnref{grantSJ}}
%\ead{svante.janson@math.uu.se}
%\ead[url]{http://www2.math.uu.se/$\sim$svante}
%
%\author{Takis Konstantopoulos\fnref{grantTK}}
%\ead{takiskonst@gmail.com}
%\ead[url]{http://www2.math.uu.se/$\sim$takis}
%\address{Department of Mathematics, Uppsala University, P.O.\ Box 480,
%SE-75106 Uppsala, Sweden}
%
%\author{Linglong Yuan\fnref{LY}}
%\ead{yuanlinglongcn@gmail.com}
%\ead[url]{http://linglongyuan.weebly.com}
%\address{Institut f\"ur Mathematik,
%Fachbereich 08 -- Physik, Mathematik und Informatik, Staudingerweg 9,
%55099 Mainz, Germany}
%
%\fntext[grantSJ]{Partly supported by the Knut and Alice Wallenberg Foundation}
%\fntext[grantTK]{(Corresponding author)
%Supported by Swedish Research Council grant 2013-4688}
%\fntext[LY]{The work of this author was done while he was a postdoctoral
%researcher at Uppsala University}

\begin{abstract}
A random vector $X=(X_1,\ldots,X_n)$ with the $X_i$ taking
values in an arbitrary measurable space $(S, \SS)$ is exchangeable
if its law is the same as that of $(X_{\sigma(1)}, \ldots, X_{\sigma(n)})$ 
for any permutation $\sigma$. We give an alternative
and shorter  proof of the representation
result (Jaynes \cite{Jay86} and Kerns and Sz\'ekely \cite{KS06})
stating that the law of $X$ is a mixture of
product probability measures  with respect to a signed mixing measure.
The result is ``finitistic'' in nature meaning that it is a matter
of linear algebra for finite $S$. The passing from finite $S$ to an
arbitrary one may pose some measure-theoretic difficulties which
are avoided by our proof.
The mixing signed measure is not unique (examples are given),
but we pay more attention to the one constructed in the proof 
(``canonical mixing measure'') by pointing out some of its
characteristics. 
The mixing measure is, in general, defined on
the space of probability measures on $S$; but for $S=\R$,
one can choose a mixing measure on $\R^n$. 

\vspace*{2mm}

\vspace*{2mm}
\ams{60G09}{60G55,62E99}

\end{abstract}

\section{Introduction}
The first result that comes to mind when talking about exchangeability
is de Finetti's
%\footnote{Or, more precisely, 
%the de Finetti/Hewitt/Savage/Ryll-Nardzewski theorem} 
theorem concerning sequences $X=(X_1, X_2,\ldots)$
of random variables with values in some space $S$
and which are invariant under permutations of finitely many coordinates. 
This remarkable theorem \cite[Theorem 11.10]{KALL02} states that
the law of such a sequence is a mixture of product measures:
let $S^\infty$ be the product of countably many copies of $S$
and let $\pi^\infty$ be the product measure on $S^\infty$ with marginals $\pi
\in \PP(S)$ (the space of probability measures on $S$); then
\[
\P(X \in \cdot) = \int_{\PP(S)} \pi^\infty(\cdot)\, \nu(d\pi),
\]
for a uniquely defined probability measure $\nu$ which we call
a mixing (or directing) measure. 
%TRUE1
%We say that the law of $X$ is a {\em true mixture} of product measures
%in order to emphasize that $\nu$ is a probability measure. 
%In this paper, we shall deal with situations where $\nu$ is a {\em
%signed measure},
%in which case we shall simply use the term mixture.

In Bayesian language, this says that any exchangeable random
sequence is obtained by first picking a probability distribution $\pi$ from
some prior (probability distribution on the space
of probability distributions) and then letting
the $X_i$ to be i.i.d.\ with common law $\pi$. 
As Dubins and Freedman \cite{DF79} show, de Finetti's theorem 
does not hold for an arbitrary measurable space $S$. Restrictions
are required. One of the most general cases for which the theorem
does hold is that of a Borel space $S$, i.e., a space which
is isomorphic (in the sense of existence of a measurable
bijection with measurable inverse) to a Borel subset of $\R$.
Indeed, one of the most elegant proofs of the theorem
can be found in Kallenberg \cite[Section 1.1]{Ka05}
from which it is evident that the main ingredient is the ergodic
theorem and that the Borel space is responsible for the existence
of regular conditional distributions.

For finite dimension $n$, however, things are different.
Let $S$ be a set together with a $\sigma$-algebra $\SS$,
and let $X_1, \ldots, X_n$ be measurable functions from a
measure space $(\Omega, \FF)$ into $(S, \SS)$.
Under a probability measure $\P$ on $(\Omega, \FF)$,
assume that $X=(X_1, \ldots, X_n)$ 
is such that 
$\sigma X := (X_{\sigma(1)}, \ldots, X_{\sigma(n)})$ has
the same law as $(X_1, \ldots, X_n)$
for any permutation $\sigma$ of $\{1,\ldots, n\}$,
i.e., that $\P(\sigma X \in B) = \P(X \in B)$ for all $B \in \SS^n$,
where $\SS^n$ is the product $\sigma$-algebra on $S^n$.
In such a case, we say that $X$ is $n$-exchangeable (or simply exchangeable).

\begin{example}
\label{Ex1}
Simple examples show that a finitely exchangeable random vector
%TRUE2
may not be a mixture of product measures.
For instance, take $S=\{1,\ldots,n\}$, with $n\ge2$,
and let $X=(X_1, \ldots, X_n)$ take values in $S^n$ such
that $\P(X=x) = 1/n!$ when $x=(x_1, \ldots, x_n)$ is a permutation
of $(1,\ldots, n)$, and $\P(X=x)=0$ otherwise.
Clearly, $X$ is $n$-exchangeable.
%TRUE3
Suppose that the law of $X$ is a mixture of product measures.
Since the space of probability measures $\PP(S)$ can naturally be
identified with the set $\Sigma_n:=\{(p_1,\ldots,p_n)\in \R^n:\,
p_1, \ldots, p_n \ge 0,\, p_1+\cdots+p_n=1\}$, the assumption
%TRUE4
that the law of $X$ is a mixture of product measures
is equivalent to the following: there is a random variable $p=(p_1,\ldots,p_n)$
with values in $\Sigma_n$ such that
$\P(X=x) = \E [\P(X=x | p)]$, where $\P(X=x | p) =
p_{x_1}\cdots p_{x_n}$ for all $x_1,\ldots,x_n \in S$.
But then, for all $i\in S$, $0=\P(X_1=\cdots=X_n=i) = \E [p_i^n]$,
implying that $p_i=0$, almost surely, for all $i \in S$, an obvious
contradiction.
\end{example}

However, Jaynes \cite{Jay86} showed that (for the $|S|=2$ case)
there is mixing provided that signed measures are allowed; see
equation \eqref{xi} below.
Kerns and Sz\'ekely \cite{KS06} observed that the Jaynes result
can be generalized to {\em an arbitrary} measurable space $S$,
but the proof in \cite{KS06} 
requires some further explicit arguments.
In addition, \cite{KS06} uses a non-trivial algebraic result 
without a proof.
Our purpose in this note is to give an alternative, shorter, and 
rigorous proof of 
the representation result (see Theorem \ref{mth} below)
but also to briefly discuss some consequences and
open problems (Theorem \ref{corob} and Section \ref{Seclast}).
An an independent proof of an
algebraic result needed in the proof of Theorem \ref{mth}
is presented in the appendix as Theorem \ref{hompol}.
To the best of our knowledge, the proof is new and, possibly,
of independent interest.

\begin{theorem}[Finite exchangeability representation theorem]
\label{mth}
Let $X_1, \ldots, X_n$ be random variables on some  probability space
$(\Omega, \FF, \P)$ with values in a measurable space $(S, \SS)$.
Suppose that the law of $X=(X_1, \ldots, X_n)$ is exchangeable.
Then there is a signed measure $\xi$ on $\PP(S)$
\begin{equation}
\label{xi}
\P(X \in A) = \int_{\PP(S)} \pi^n(A)\, \xi(d\pi),
\quad A \in \SS^n,
\end{equation}
where $\pi^n$ is the product of $n$ copies of $\pi \in \PP(S)$.
\end{theorem}

We stress that the theorem does not ensure uniqueness of $\xi$.
\begin{example}
\label{nonu}
To see this in an example, consider Example \ref{Ex1} 
with $n=2$, that is, let $S=\{1,2\}$ and let $X=(X_1, X_2)$ 
take values $(1,2)$, $(2,1)$, $(1,1)$, $(2,2)$ with probabilities
$1/2$, $1/2$, $0$, $0$, respectively.
We identify $\PP(S)$ with the interval $[0,1]$, via $\pi\{1\}=p$,
$\pi\{2\}=1-p$, for $\pi \in \PP(S)$. We give
three different signed measures that can be used in the representation.
\\
(i)  Let $\xi$ be the signed measure on $[0,1]$ defined by
\[
\xi = -\frac{1}{2} \delta_0 - \frac{1}{2} \delta_1 + 2 \delta_{1/2}.
\]
Then $\P(X_1=1, X_2=2) = \int_{[0,1]} p(1-p)\, \xi(dp) = 1/2 =
\P(X_1=2, X_1=1)$, 
while $\P(X_1=1, X_2=1) = \int_{[0,1]} p^2 \, \xi(dp)=0
= \P(X_1=2, X_2=2)$.
\\
(ii) Let
\[
\xi = -\frac{5}{8} \delta_0
-\frac{5}{8} \delta_1
+\frac{9}{8} \delta_{1/3}
+\frac{9}{8} \delta_{2/3}.
\]
Again, $\int_{[0,1]} p(1-p)\, \xi(dp) =1/2$,
$\int_{[0,1]} p^2\, \xi(dp) = \int_{[0,1]} (1-p)^2\, \xi(dp) =0$.
\\
(iii)  
Let $\xi$ be a signed measure with density
\[
f(p) := -\frac{7}{2} \cdot \1_{p \le 1/3 \text{ or } p \ge 2/3}
+ 10 \cdot \1_{1/3 < p < 2/3}.
\]
% and let $\xi(dp) = f(p) dp$.
We can easily see that $\int_0^1 f(p) dp =1$, $\int_0^1 p^2 f(p)dp=0$,
$\int_0^1 p(1-p) f(p) dp=1/2$.

\end{example}

\begin{remark}
\label{rem1}
The difference between this situation and the one in de Finetti's setup
is that a finitely exchangeable random vector 
$(X_1, \ldots, X_n)$  is not necessarily extendible to an infinite
sequence $(X_1, \ldots, X_n, X_{n+1}, \ldots)$ that is exchangeable.
(See Examples \ref{lastex1} and \ref{lastex2} below.)
If it were, then the signed measure $\xi$
could have been chosen as a probability measure
(and would then have been unique).
The question of extendibility of an $n$-exchangeable $(X_1, \ldots, X_n)$
to an $N$-exchangeable $(X_1, \ldots, X_N)$, for some $N > n$ (possibly
$N=\infty$) is treated in the  sequel paper \cite{KYextend}
that strongly uses the framework and results of the present paper.
Assuming such extendibility,
Diaconis and Freedman \cite{DIA77, DF80} show that the total variation 
distance of an $n$-exchangeable probability measure on $S^n$ from the set
%TRUE5
of mixtures of product probability measures
is at most $n(n-1)/N$ when $S$ is an infinite set
(and at most $2|S|n/N$ if $S$ is finite.)

\end{remark}

When $S=\R$, it is possible to say more than in Theorem \ref{mth}:
\begin{theorem}
\label{corob}
Let $(X_1, \ldots, X_n)$ be an $n$-exchangeable random vector in
$\R^n$, endowed with the Borel $\sigma$-algebra.
Then there is a bounded signed measure
$\eta(d\theta_1, \ldots, d\theta_n)$, 
%supported on a cone in $\R^n$, 
such that
$\P(X\in A) = \int_{\R^n} \pi_{\theta_1, \ldots, \theta_n}^n(A)\, \eta
(d\theta_1, \ldots, d\theta_n)$,
where $\pi_{\theta_1, \ldots,\theta_n}$ is an element of $\PP(S)$ depending
measurably on the $n$ parameters $(\theta_1, \ldots, \theta_n)$.
\end{theorem}

\section{Preliminaries and notation}
We make use of the following notations and terminology in the paper.
If $S$ is a set with a $\sigma$-algebra $\SS$, then $\PP(S)$
is the set of probability measures on $(S, \SS)$. The space $\PP(S)$
is equipped with the $\sigma$-algebra generated by sets
of the form $\{\pi \in \PP(S):\, \pi(B) \le t\}$, $B \in \SS$,
$t\in \R$.
We shall write $\PP(S, \SS)$ if we wish to emphasize the role of 
the $\sigma$-algebra.\footnote{In other words,
when we write $\PP(S, \SS)$, we mean that $\PP(S)$ is given the 
$\sigma$-algebra generated by sets
$\{\pi \in \PP(S):\, \pi(B) \le t\}$, $B \in \SS$, $t \in \R$.}
Similarly, $\MM(S)$ or $\MM(S, \SS)$
will be the space of bounded signed measures, equipped
with a $\sigma$-algebra as above.
%We call a signed measure bounded  if its total variation measure
%is bounded.
In particular, $\MM(\PP(S))$ is the space of bounded signed measures
on $\PP(S,\SS)$.
A random measure on $S$ is a measurable mapping from $\Omega$ into
$\PP(S)$ and a random signed measure on $\PP(S)$
%such as the quantity $\psi_X$ of Theorem \ref{mth}, 
is a measurable mapping from $\Omega$ into $\MM(\PP(S))$.
The measure $\xi$ in Theorem \ref{mth} is an element of $\MM(\PP(S))$.
The delta measure $\delta_a$ at a point $a\in S$ is, as usual,
the set function $\delta_a(B) := \1_{a \in B}$, 
$B \subset S$. A finite point measure is a finite linear combination
of delta measures where the coefficients are nonnegative integers.
We let $\NN(S)$ be the set of finite point measures on $S$ and 
$\NN_n(S)$ the set of point measures $\nu$ such that $\nu(S)=n$.
The symbol $(\nu)!$ is defined as
\[
(\nu)! := \prod_{a \in S} \nu\{a\} !
\]
where $\nu\{a\}$ is the value of $\nu$ at the singleton $\{a\}$ and
where the product is over the support of $\nu$ %because, as usual, $0!:=1$.
($0!:=1$).
The symbol $\SS^n$ stands for the product $\sigma$-algebra on $S^n$.
If $\pi \in \PP(S)$ then $\pi^n \in \PP(S^n)$ is the product measure
of $\pi$ with itself, $n$ times.
If $x=(x_1, \ldots, x_n) \in S^n$ then the {\em type} of $x$ 
is the element $\epsilon_x$ of $\NN_n(S)$ defined by 
\[
\epsilon_x:= \sum_{i=1}^n \delta_{x_i}.
\]
%We call $\epsilon_x$ the {\em type} of $x$.
The set $S^n(\nu) \subset S^n$ is defined by,
for $\nu\in\NN_n(S)$,
\[
S^n(\nu) := \{y \in S^n:\, \epsilon_y = \nu\}.
\]
It is a finite set with cardinality 
\[
\binom{n}{\nu} := \frac{n!}{(\nu)!}.
\]
We let $\bm u_\nu$ be the uniform probability measure on $S^n(\nu)$,
that is, 
\[
\bm u_\nu = \binom{n}{\nu}^{-1} \sum_{z \in S^n(\nu)} \delta_z.
\]
If $\SS$ is too coarse, then $S^n(\nu)$ may not belong to $\SS^n$. 
This is not a problem when $\SS$ is, say, the Borel $\sigma$-algebra of a 
Hausdorff space, but we wish to prove the result without any 
topological assumptions.
Moreover, notice that
\begin{equation}
\label{Sdecomp}
S^n = \bigcup_{\nu \in \NN_n(S)} S^n(\nu),
\end{equation}
since $y \in S^n(\epsilon_y)$ for all $y \in S^n$. The sets in the
union are pairwise disjoint because $S^n(\nu) \cap S^n(\nu') = \varnothing$
if $\nu$ and $\nu'$ are distinct elements of $\NN_n(S)$.
If $\sigma$ is a permutation of $\{1,\ldots,n\}$ and $x \in S^n$,
then $\sigma x := (x_{\sigma(1)}, \ldots, x_{\sigma(n)})$.

\section{Proof of the finite exchangeability representation theorem}

%\paragraph{\bf Step 1}
By exchangeability, for any $B \in \SS^n$,
\[
\P(X \in B) =\frac{1}{n!}
\sum_\sigma \P(\sigma X\in B)
=\E \frac{1}{n!} \sum_\sigma \delta_{\sigma X}(B) = \E \bm U_X(B)
\]
where the sum is taken over all permutations $\sigma$ of $\{1,\ldots,n\}$,
and where
\[
\bm U_x := \frac{1}{n!} \sum_\sigma \delta_{\sigma x},
\quad x \in S^n.
\]
Notice that the map $x \mapsto \bm U_x$ is a measurable 
function from $(S^n, \SS^n)$ into $\PP(S^n, \SS^n)$, and, since
$X$ is a measurable function from $(\Omega, \FF)$ into $(S^n,\SS^n)$,
we have that  $\bm U_X$ is a random element of $\PP(S^n,\SS^n)$.
The mean measure $\E \bm U_X$ is the probability law of $X$.

%\paragraph{\bf Step 2}
Forgetting temporarily that our original space is $S$,
consider a finite set $T$ and let $Q$ be an exchangeable
probability measure on $T$.          
Then, for all $\nu \in \NN_n(T)$, $Q$ assigns the same value to
every singleton of  $T^n(\nu)$.
Hence 
\begin{equation}\label{star=}
Q = \sum_{\nu \in \NN_n(T)} Q(T^n(\nu)) \, \bm u_\nu,
\end{equation}
where $\bm u_\nu$ is the uniform probability measure on $T^n(\nu)$.
In particular, let $Q=\pi^n$, where $\pi \in \PP(S)$.
It is easy to see (multinomial distribution) that
\[
\pi^n(T^n(\nu)) = \binom{n}{\nu}\, \pi^\nu,
\]
where $\pi^\nu := \prod_{a \in T} \pi\{a\}^{\nu\{a\}}$
(adopting the convention $0^0=1$).
Specialize further by letting $\pi=\lambda/n$ where $\lambda \in \NN_n(T)$.
Canceling a factor, \eqref{star=} gives
\begin{equation}
\label{plus=}
\lambda^n = \sum_{\nu \in \NN_n(T)}  \binom{n}{\nu}\, \lambda^\nu
\, \bm u_\nu.
\end{equation}
Let $W$ be a matrix with entries $W(\lambda, \nu) := \binom{n}{\nu}\lambda^\nu$,
$\lambda, \nu \in \NN_n(T)$.
This is essentially the multinomial Dyson matrix;
see \eqref{Wdef} in the Appendix and the discussion therein.
Let $M$ be the inverse of $W$; see \eqref{minus} in the Appendix.
From \eqref{minus} and \eqref{plus=} we have
\begin{equation}
\label{unu=}
\bm u_\nu = \sum_{\lambda \in \NN_n(T)} M(\nu,\lambda) \, \lambda^n.
\end{equation}
This is an equality between measures on $T^n$.

%\paragraph{\bf Step 3}
Specialize further by letting $T=[n]:=\{1,\ldots,n\}$.
Fix $x=(x_1,\ldots,x_n) \in S^n$.
Define $\phi_x :[n] \to S$ by $\phi_x(i) = x_i$, $i=1,\ldots,n$.
This induces a linear map $\MM([n]) \to \MM(S)$, also denoted by $\phi_x$,
by the formula $\phi_x(\delta_i)=\delta_{\phi_x(i)} = \delta_{x_i}$,
$i \in [n]$,
and extended by linearity:
$\phi_x(\sum_{i=1}^n c_i \delta_i) = \sum_{i=1}^n c_i \phi_x(\delta_i)$.
Define $\phi_x^n :[n]^n \to S^n$ by
$\phi_x^n(i_1, \ldots, i_n) = (\phi_x(i_1), \ldots, \phi_x(i_n))$.
This again induces a linear map $\MM([n]^n) \to \MM(S^n)$, also
denoted by $\phi_x^n$, by the formula $\phi_x^n (\delta_j)
= \delta_{\phi_x^n(j)}$, $j \in [n]^n$, and extended by linearity.
We can then easily show that $\phi_x^n(\mu_1 \times \cdots \times \mu_n)
= \phi_x(\mu_1) \times \cdots \times \phi_x(\mu_n)$ for any 
$\mu_1,\ldots,\mu_n \in \MM([n])$.
Let now $\nu_n$ be the measure on $[n]$ with $\nu_n\{i\}=1$, $i=1,\ldots,n$.
Let $\NN_n(n):= \NN_n([n])$. Then \eqref{unu=} yields
\[
\bm u_{\nu_n} = \sum_{\lambda\in \NN_n(n)} M_n(\lambda) \lambda^n,
\]
where $M_n(\lambda) := M(\nu_n, \lambda)$.
It is easy to see that
\[
\bm u_{\nu_n} = \frac{1}{n!} \sum_\sigma \delta_{\sigma \iota} = \bm U_\iota,
\]
where $\iota :=(1,\ldots,n)$ and where the sum
is taken over all permutations $\sigma$ of $[n]$.
The last two displays are equalities between measures
on $\{1,\ldots,n\}^n$.

%\paragraph{\bf Step 4}
For each $x\in S^n$ define 
\begin{equation}\label{psix}
\psi_x 
:= \sum_{\lambda \in \NN_n(n)} n^n M_n(\lambda) \delta_{\phi_x(\lambda/n)},
\end{equation}
a signed measure on $\PP(S)$.  We are going to show that
\\[1mm]
(i) $\int_{\PP(S)} \tau^n \psi_x(d\tau) = \bm U_x$,
\\
(ii) $x \mapsto \psi_x$ is a measurable map $S \to \MM(\PP(S))$.
\\[1mm]
To show (i), observe, directly from the definition of $\psi_x$, that
\begin{equation}
\int_{\PP(S)} \tau^n \psi_x(d\tau) =
\sum_{\lambda \in \NN_n(n)} n^n M_n(\lambda)
\phi_x(\lambda/n)^n.
\end{equation}
But
$\phi_x(\lambda/n)^n = \phi_x^n((\lambda/n)^n) = n^{-n} \phi_x^n(\lambda^n)$
and so
\begin{multline}
\int_{\PP(S)} \tau^n \psi_x(d\tau) =
\sum_{\lambda \in \NN_n(n)}  M_n(\lambda)
\phi_x^n(\lambda^n)
= \phi_x^n \left(\sum_{\lambda \in \NN_n(n)} M_n(\lambda) \lambda^n \right)
= \phi_x^n(\bm u_{\nu_n}) = \phi_x^n(\bm U_\iota)
\\ 
= \frac{1}{n!} \sum_\sigma \phi_x^n (\delta_{\sigma \iota})
= \frac{1}{n!} \sum_\sigma \delta_{\phi_x^n(\sigma\iota)}
= \frac{1}{n!} \sum_\sigma \delta_{\sigma x} = \bm U_x.
\label{Nsmall=}
\end{multline}
To show (ii), we first observe that  
$\phi_x(\lambda/n) = \sum_{i=1}^n\frac{\lambda_i}{n}
\delta_{x_i}$ and that the maps $x \mapsto \delta_{x_i}$,
$S^n\to \PP(S)$, are measurable.
It then follows that $x \mapsto \phi_x(\lambda/n)$,
$S^n \to \PP(S)$, is measurable.		
Also, the map $\mu \mapsto \delta_\mu$, $\PP(S) \to \MM(\PP(S))$, is measurable,
Since composition of measurable functions is measurable,
we have that $x \mapsto \delta_{\phi_x(\lambda/n)}$ is a measurable function
from $S$ into $\MM(\PP(S))$.

%\paragraph{\bf Step 5}
Since $x \mapsto \psi_x$ is measurable
we have that $\psi_X$ is a random element of $\MM(\PP(S))$
and thus $\xi:=\E \psi_X$ is a well-defined element of $\MM(\PP(S))$. 
Note that $\psi_X$ is
also bounded so there is no problem with taking the expectation.
On the other hand, since $\bm U_X = \int_{\PP(S)} \tau^n \psi_X(d\tau)$,
a.s.,  and since $\E \bm U_X$ is the probability distribution of $X$,
the assertion \eqref{xi} follows with $\xi = \E \psi_X$.

\section{Additional results and applications}
\label{Seclast}
%The canonical mixing measure and proof of Theorem \ref{corob}

\subsection{Proof of Theorem \ref{corob}}\label{SSpf2}
%Since $\NN_n(T_x) \subset \NN_n(S)$,
%notice that 
By \eqref{psix},
the integration on the left-hand side of
\eqref{Nsmall=} actually takes place over the set
$\{\phi_x(\lambda/n):\lambda\in\NN_n(n)\}$;
since $\phi_x(\lambda/n)=\frac{1}n\phi_x(\lambda)$ and
$\phi_x(\lambda)\in\NN_n(S)$ when $\lambda\in\NN_n(n)$, 
it follows that it suffices to integrate over
$\frac{1}{n} \NN_n(S) := \{\frac{1}{n} \lambda:\, \lambda \in \NN_n(S)\}
\subset \PP(S)$.
Let $S=\R$. Then $\frac{1}{n} \NN_n(\R)$ is a measurable subset of $\PP(\R)$.
Then, for any Borel subset $B$ of $\R^n$, 
\[
\P(X \in B) = \int_{\frac{1}{n} \NN_n(\R)} \tau^n(B) \, \xi(d\tau).
\]
We can write
$\NN_n(\R) = \bigcup_{d=1}^n \NN_{n,d}(\R)$
where $\NN_{n,d}(\R)$ is the set of
all point measures $\nu$ on $\R$ with total mass equal to $n$
and support of size $d$.
The set $\NN_{n,n}(\R)$ can be identified with all
$(x_1, \ldots, x_n) \in \R^n$ such that $x_1 < \cdots < x_n$.
This is an open cone $C$. Each of the other sets, $\NN_{n,d}(\R)$,
$d=1,\ldots, n-1$, corresponds to a particular subset of the
boundary of $C$. Therefore, we can replace the integration by
integration on a cone of $\R^n$. 
\qed

This result tells us that in order to represent an $n$-exchangeable
random vector in $\R^n$ as an integral against an unknown signed measure
we may as well search for a signed measure on a space of dimension 
$n$ rather than on the space $\PP(\R)$.
Of course, we chose $S=\R$ in Theorem \eqref{corob} as a matter of
convenience. A similar result can be formulated more generally.

\subsection{A more explicit formula for the mixing measure $\psi_x$}
We claim that $\psi_x$, defined by \eqref{psix}, also is given by
\begin{equation}
  \label{psix'}
  \psi_x = \sum _{\lambda \in \NN_n(T)} n^n M(\epsilon_x,\lambda) \,
\delta_{\lambda/n},
\end{equation}
for any finite set $T$ such that $x_1,\dots,x_n\in T$.

Indeed, if we temporarily denote the right-hand side of \eqref{psix'} by
$\psi'_x$, then, using \eqref{unu=},
\begin{equation}\label{psi'a}
  \begin{split}
\int_{\PP(S)} \tau^n\, \psi'_x(d\tau)
%= \sum_{\lambda \in \NN_n(T_x)} n^n M(\epsilon_x,\lambda)
%\int_{\PP(S)} \tau^n\, \delta_{\lambda/n}(d\tau)
%\\
= \sum_{\lambda \in \NN_n(T)} n^n M(\epsilon_x,\lambda)
\, (\lambda/n)^n
%= \sum_{\lambda \in \NN_n(T)} M(\epsilon_x,\lambda) \lambda^n
=\bm u_{\epsilon_x}
= \bm U_x,  	
  \end{split}
\end{equation}
where the last equality follows from the definitions. 
By \eqref{psi'a} and \eqref{Nsmall=},
\begin{equation}
  \label{psipsi}
\int_{\PP(S)} \tau^n\, \psi'_x(d\tau)
=\int_{\PP(S)} \tau^n\, \psi_x(d\tau). 
\end{equation}
Now note that if $\lambda\in\NN(n)$, then
$\phi_x(\lambda/n)\in\frac{1}n\NN(T)$, c.f.\ Section \ref{SSpf2},
and thus the measures $\psi_x$ and $\psi'_x$ both are supported on
$\frac{1}n\NN(T)$. 
Moreover,
the measures $\bm u_\nu$, $\nu\in\NN(T)$, are linearly dependent
and form thus a basis in a linear space of dimension $|\NN(T)|$. By
\eqref{plus=} and \eqref{unu=}, the measures $\lambda^n$,
$\lambda\in\NN(T)$,
span the same space, so they form another basis and are therefore linearly
independent. 
Hence, the measures $\tau^n$, $\tau\in\frac{1}n\NN(T)$,
are linearly independent.
Consequently, the equality \eqref{psipsi}
implies that $\psi'_x=\psi_x$, as claimed.

In particular, we can in \eqref{psix'} always choose
$T=T_x:=\{x_1,\dots,x_n\}$. 

\subsection{The canonical mixing measure}\label{canon}
This is the particular signed measure $\xi$ constructed as the mean measure 
of the random signed measure $\psi_X$, given by \eqref{psix}.

For example, let $n=2$.
An easy computation of the matrix $W$ and its inverse $M$ when $T=\{1,2\}$
shows that
$M_2(\lambda):=M(\nu_2, \lambda) =
-1/8,\, 1/2,\, -1/8$ 
when $\lambda = 2\delta_{1},\, \delta_{1}+\delta_{2},\,
2\delta_{2}$, respectively.
So, we have
\begin{equation}
  \label{psi2}
\psi_{(X_1,X_2)} = -\frac12 \delta_{\delta_{X_1}}
-\frac12 \delta_{\delta_{X_2}} + 2 \delta_{(\delta_{X_1}+\delta_{X_2})/2}.
\end{equation}
Hence, 
\begin{equation}\label{xi2}
\xi= 2 \P((\delta_{X_1}+\delta_{X_2})/2\in \cdot)
-  \P(\delta_{X_1} \in\cdot).
\end{equation}

We can also use the formula \eqref{psix'} for $\psi_X$, taking $T=T_X$.
Let $d(X)$ be the cardinality of the set $\{X_1, \ldots, X_n\}$.
On the event $\{d(X)=d\}$, for
some $d \in \{1, \ldots,n\}$, the variable $\lambda$ in the summation in
\eqref{psix'} 
ranges over a set of cardinality $\binom{n+d-1}{n}$.

For example, let again $n=2$, assume that the event
$\{d(X)=2\}=\{X_1\neq X_2\}$ is measurable
and let $p := \P(d(X)=2)$.
On the event $\{d(X)=1\}$, we have $T_X=\{X_1\}$ and so $\NN_2(T_X)
=\{2\delta_{X_1}\}$. Hence 
$\psi_X = \delta_{\delta_{X_1}}$.
On the event $\{d(X)=2\}$, we obtain \eqref{psi2}.
This yields the formula, obviously equivalent to \eqref{xi2}, 
\begin{multline*}
\xi=(1-p) \P(\delta_{X_1} \in \cdot \mid d(X)=1) 
- p\, \P(\delta_{X_1} \in \cdot \mid d(X)=2) 
\\
+ 2p \P((\delta_{X_1}+\delta_{X_2})/2\in \cdot \mid d(X)=2).
\
\end{multline*}

\subsection{Moment functional}
The $k$-th moment functional of a mixing signed measure $\xi$ is defined by
\[
C_k(B_1, \ldots, B_k) := \int_{\PP(S)} \pi(B_1) \cdots \pi(B_k)\, \xi(d\pi).
\]
If $k \le n$, then, from \eqref{xi},
\[
C_k(B_1, \ldots, B_k) = \P(X\in B_1\times \cdots \times B_k).
\]
This means that any mixing measure $\xi$ will have 
the same $C_k$ for all $k \le n$.
But if $k > n$, then $C_k(B_1, \ldots, B_k)$ may be negative and will
depend on the choice of $\xi$. For the canonical $\xi$, we have,
using \eqref{psix},
\begin{align*}
C_k(B_1, \ldots, B_k) 
&= \E \sum_{\lambda \in \NN_n(T_X)} n^n M(\epsilon_X, \lambda)
(\lambda/n)(B_1) \cdots (\lambda/n)(B_k)
\\
&= \E \sum_{\lambda \in \NN_n(T_X)}
 M(\epsilon_X, \lambda)
\lambda(B_1) \cdots \lambda(B_k).
\end{align*}

\subsection{Laplace functional}
Define next the Laplace functional of the 
canonical mixing measure $\xi$ by
\[ 
\Lambda(f) := \xi\bigg[\exp \left(-\int_S f(a) \pi(da)\right)\bigg]
= \int_{\PP(S)}  e^{-\int_S f(a) \pi(da)}\, \xi(d\pi),
\]
for $f: S \to \R_+$  measurable.
We obtain
\begin{equation}
  \begin{split}
\Lambda(f) 
&= 
\E \int_{\PP(S)}  e^{-\int_S f(a) \pi(da)}\, \psi_X(d\pi)
=
 \E \sum_{\lambda \in \NN_n(n)} n^nM_n(\lambda) \,
e^{-\frac{1}{n} \int_S f d \phi_x(\lambda/n)}.
\\&
=
 \E \sum_{\lambda \in \NN_n(n)} n^nM_n(\lambda) \,
e^{-\frac{1}{n} \sum_{i=1}^n \lambda_i f(x_i)}.
  \end{split}
\end{equation}
For example, if $n=2$, by the values of $M_2(\lambda)$ in Section
\ref{canon} and symmetry,
\[
\Lambda(f) = 2 \E \big[e^{-f(X_1)/2} e^{-f(X_2)/2}\big] - \E e^{-f(X_1)}.
\]

\subsection{Extendibility}
It is easy to see that an $n$-exchangeable random vectors may not be
extendible to an $N$-exchangeable random vector (see Remark \ref{rem1}). 
Here are two easy examples.
\begin{example}
\label{lastex1}
As in Example \ref{nonu},
with $S=\{1,2\}$, the random variable $(X_1,X_2)$ taking values in $S^2$,
such that $\P(X=(1,2))=\P(X=(2,1))=1/2$,
cannot be extended to an
exchangeable random variable $(X_1,X_2,X_3)$ with values in $S^3$.
\end{example}
\begin{example}
\label{lastex2}
Let $(X_1, X_2)$ be a Gaussian vector with $\E X_1=\E X_2 =0$,
$\E X_1^2 = \E X_2^2 = 1+\epsilon$, $\epsilon \in (0,1)$,
$\E X_1 X_2 = -1$.
If this were extendible to an exchangeable vector $(X_1, X_2, X_3)$,
then we would have had $\E X_1 X_3 = \E X_2 X_3 = -1$.
An easy calculation shows that the matrix
$\begin{pmatrix} 1+\epsilon & -1 & -1 \\ -1 & 1+\epsilon & -1 \\
-1 & -1 & 1+\epsilon\end{pmatrix}$ is not positive definite
when $\epsilon<1$, % is a sufficiently small positive number 
and thus
fails to be the covariance matrix of $(X_1, X_2, X_3)$.
\end{example}
In \cite{KYextend} we give a necessary and sufficient condition
for extendibility.

\subsection{Some applications}
Consider the following statement.
\begin{lemma}
Let $(S, \SS)$ be a measurable space, $n$ a positive integer,
and $f: S^n \to \R$ a bounded measurable function
such that $\int_{S^n} f(x_1, \ldots, x_n) P(dx_1) \cdots P(dx_n)=0$
for any probability measure $P$ on $(S, \SS)$.
Then $\int_{S^n} f d Q =0$ for any exchangeable probability measure
$Q$ on $S^n$.
\end{lemma}
Although this can be proven by other methods, it follows 
immediately from Theorem \ref{mth}.

For a more practical application, we
refer to the paper of Kerns and Sz\'ekely \cite{KS06}
for an application of Theorem \ref{mth} to the Bayesian consistency
problem. In situations where one has a fixed number $n$ of 
unordered samples, one can refer to de Finetti's theorem
in order to prove consistency of standard Bayesian estimators.
The theorem assumes that the samples come from random vectors
that are infinitely extendible (otherwise, de Finetti's theorem
does not hold). As pointed out in \cite{KS06},
the result of Theorem \ref{mth} still allows proving 
Bayesian consistency.

For a practical application of the representation result
to the Bayesian properties of normalized maximum likelihood,
see Barron, Ross and Watanabe \cite{BRW14}.

\subsection{Open problems}
Estimate the size (in terms of total variation) of
the signed measure $\xi$ in the representation \eqref{xi}.
How does this behave as a function of the dimension $n$?
What is the best bound?

While this paper deals with a probability measure on $S^n$ that
is invariant under all $n!$ permutations of coordinates, 
it is natural to ask
if there is a representation result for measures that are
invariant under a subgroup of the symmetric group.
%Natural subgroups do arise in applications. 
%Also, what if the measure is not finite or even signed?

\appendix
\section{Invertibility of the multinomial Dyson matrix}
Let $T$ be a finite set, say $T=\{1,\ldots,d\}$.
With the notation established in the introduction, 
\begin{equation}
\label{Wdef}
W(\lambda, \nu) := \binom{n}{\nu} \, \lambda^\nu,
\quad \lambda, \nu \in \NN_n(T).
\end{equation}
The matrix $[n^{-n} W(\lambda,\nu)]$  on $\NN_n(T)$
is referred to as the multinomial Dyson matrix \cite{SHMR86}. In fact,
$n^{-n} W(\lambda,\nu)$ is the 1-step transition probability
of a multitype Wright-Fisher Markov chain with state space $\NN_n(T)$.
%\\{\red Need reference for multitype Wright-Fisher please!}\\
This chain is defined as follows (see, e.g., \cite{DAW}). 
There is a population
of always constant size $n$. Individuals in this population are of
different types; the set of types is $T$.
Given the vector $\lambda=(\lambda_1, \ldots, \lambda_d)$ of type counts
of the population currently, select an individual at random and copy its type;
do this selection $n$ times, independently.
Then the probability that the vector of type counts
changes from $\lambda$ to $\nu$ is exactly equal to $n^{-n} W(\lambda,\nu)$.
Shelton {\em et al.} \cite{SHMR86} show that the matrix $W$ is invertible, 
i.e., that
there is a matrix $M$ on $\NN_n(T)$ such that
\begin{equation}
\label{minus}
\sum_{\lambda \in \NN_n(T)} M(\nu,\lambda)\, W(\lambda,\nu')
= \1_{\nu=\nu'}, \quad \nu, \nu' \in \NN_n(T).
\end{equation}
The inverse matrix $M$ can be expressed explicitly in terms
of sums involving binomial coefficients and signed Stirling numbers of
the first kind \cite[eq.\ (25)]{MOAK90}.

For a direct proof of the invertibility of $W$ that avoids explicit
computations we proceed as follows. The columns of $W$
are linearly independent if and only if the only numbers
$c(\lambda), \lambda \in \NN_n(T)$, for which
\[
\sum_\lambda c(\lambda) \binom{n}{\nu} \lambda^\nu =0,\quad \text{for all }
\nu\in \NN_n(T)
\]
are zero.
But the last display is equivalent to
\[
0 = \sum_\nu x^\nu \sum_\lambda c(\lambda) \binom{n}{\nu} \lambda^\nu
= \sum_\lambda c(\lambda) (\lambda_1 x_1 + \cdots + \lambda_d x_d)^n, \quad \text{for all } x=(x_1,\ldots,x_d)
\in \R^d.
\]
Invertibility of $W$ thus follows from:
\begin{theorem}
\label{hompol}
Let $d, n$ be positive integers.
Let $\NN_n(d)$ be the set of all $\lambda=(\lambda_1,\ldots,\lambda_d)$
where the $\lambda_i$ are nonnegative integers such that $\sum_{i=1}^d \lambda_i=n$.
Then the polynomials
\[
p_\lambda(x) := (\lambda_1 x_1 + \cdots + \lambda_d x_d)^n, \quad\lambda \in \NN_n(d),
\]
are linearly independent 
and form a basis for the space $\PP_n(d)$ of homogeneous polynomials of
degree $n$ in $x_1,\ldots,x_d$.
\end{theorem}
\begin{proof}
Note that $\{x^\lambda = x_1^{\lambda_1}\cdots x_d^{\lambda_d}:\, 
\lambda \in \NN_n(d)\}$
is a basis in $\cP_n(d)$.
Let $\cQ_n(d)$ be the linear subspace of $\cP_n(d)$ spanned by
$\{p_\lambda, \lambda \in \NN_n(d)\}$. We will show that
$\cQ_n(d) = \cP_n(d)$.
The substitution $x_i \mapsto x_i+x_d$, $i=1,\ldots,d-1$,
shows that the $\cQ_n(d)$ is 
isomorphic to the subspace $\hcQ_n(d)\subset\cP_n(d)$
 spanned by the polynomials
\[
\widehat p_\lambda(x_1,\ldots,x_d)
:=(\lambda_1 x_1 + \cdots + \lambda_{d-1} x_{d-1} +n x_d)^n,
\quad \lambda \in \NN_n(d).
\]
For $0 \le j \le n$, denote by $\cP_{n,j}(d)$ the subspace of $\cP_n(d)$
consisting of polynomials that have degree in $x_d$ at most $j$.
Note that $\cP_{n,0}(d)=\cP_{n}(d-1)$ and $\cP_{n,n}(d) = \cP_n(d)$.
Let $\Delta_h$ be the difference operator acting on functions $f(t)$
of one real variable $t$ by
$\Delta_h f(t) := f(t+h)-f(t)$.
It is easy to see that, for all integers $k \ge 1$,
\begin{equation}
\label{Delta}
\Delta_{h_1} \cdots \Delta_{h_k} f(t)
= \sum_{r=0}^k (-1)^r \sum_{\substack{I \subset \{1,\ldots,k\}\\ |I|=r}}
f\big(t+\sum_{\alpha \in I} h_\alpha\big).
\end{equation}
Using \eqref{Delta} with $f(t)=t^n$ and induction on $k$ we easily obtain
\begin{equation}
\label{Delta'}
\Delta_{h_1} \cdots \Delta_{h_k} \{t^n\}
= (n)_k \, h_1 \cdots h_k t^{n-k} + r_{h_1,\ldots,h_k}(t),
\quad k=1,\ldots,n,
\end{equation}
where $t \mapsto r_{h_1,\ldots,h_k}(t)$ is a polynomial
of degree $\le n-k-1$,
whereas $(h_1, \ldots, h_k, t) \mapsto r_{h_1,\ldots,h_k}(t)$
is a homogeneous polynomial of degree $n$.
The meaning of \eqref{Delta'} for $k=n$ is that
\begin{equation}
\label{Delta'n}
\Delta_{h_1} \cdots \Delta_{h_n} \{t^n\} = n!\, h_1 \cdots h_n.
\end{equation}
Let $(i_1, \ldots, i_k)$ be a sequence with values in
$\{1,\ldots, d-1\}$.
Using \eqref{Delta}
with $f(t)= t^n$ and then setting $t=nx_d$ and
$h_1 = x_{i_1}, \ldots, h_k=x_{i_k}$ we obtain
\[
\Delta_{x_{i_1}} \cdots \Delta_{x_{i_k}} \{t^n\}\big|_{t=nx_d}
= \sum_{r=0}^k (-1)^r \sum_{\substack{I \subset \{1,\ldots,k\}\\ |I|=r}}
(nx_d+\sum_{\alpha \in I} x_{i_\alpha})^n, \quad 1 \le k \le n.
\]
For $I \subset \{1,\ldots,k\}$,
we have $(nx_d+\sum_{\alpha \in I} x_{i_\alpha})^n
= \widehat p_{\lambda}(x_1,\ldots,x_d)$
where $\lambda_j$ is the number of terms of the sequence
$(i_1, \ldots, i_k)$ that are equal to $j$, $j=1,\ldots, d$.
This implies that
(recall that
$\hcQ_n(d)$ is spanned by $\{\widehat p_\lambda, \lambda \in \NN_n(d)\}$)
the function 
$(x_1, \ldots, x_d) \mapsto \Delta_{x_{i_1}} \cdots \Delta_{x_{i_k}}
\{t^n\}\big|_{t=nx_d}$
is a polynomial in $d$ variables that belongs to $\hcQ_n(d)$.
Using this observation in \eqref{Delta'} and \eqref{Delta'n} we obtain
\begin{align}
n^{n-k}\,(n)_k \, x_{i_1} \cdots x_{i_k} x_d^{n-k}
&=\Delta_{x_{i_1}} \cdots \Delta_{x_{i_k}} \{t^n\}\big|_{t=nx_d}
- r_{x_{i_1},\ldots,x_{i_k}}(nx_d)
\nonumber
\\
& \in \hcQ_n(d) + \cP_{n,n-k-1}(d),
\qquad\qquad 1 \le k \le n-1,
\label{bel}
\end{align}
\begin{equation}
\label{beln}
n!\, x_{i_1} \ldots x_{i_n} = \Delta_{x_{i_1}} \cdots \Delta_{x_{i_n}}
\{t\}\big|_{t=nx_d}
\in \hcQ_n(d).
\end{equation}
Since every polynomial on $\cP_{n,n-k}(d)$ is a linear combination
of the monomials appearing in the left-hand side of \eqref{bel}, 
\eqref{bel} implies
\[
\cP_{n,n-k}(d) \subset \hcQ_n(d) + \cP_{n,n-k-1}(d),
\quad 1 \le k \le n-1.
\]
Similarly, every polynomial on $\cP_{n,0}(d)$ is a linear combination
of monomials as in the left-hand side of \eqref{beln}, and thus
\eqref{beln} implies
\[
\cP_{n,0}(d) \subset \hcQ_n(d).
\]
The last two displays imply that $\cP_{n,n-1}(d) \subset \hcQ_n(d)$.
Since every polynomial in $\cP_n(d)$ is a linear combination of
the monomial $x_d^n$ and a polynomial in $\cP_{n,n-1}(d)$, 
it follows that $\cP_n(d) \subset \hcQ_n(d)$ and so the proof
is completed.
\end{proof}

\subsection*{Acknowledgments}
We thank Jay Kerns and G\'abor Sz\'ekely for their comments 
and for pointing reference \cite{MOAK90} to us and
an anonymous referee for helpful comments and for 
reference \cite{BRW14}.
The second author would like to thank S{\o}ren Asmussen
for giving him the opportunity to present this work
as part of an invited talk at Aarhus University
and for his hospitality.





\end{document}